\documentclass[12pt]{amsart}
\usepackage{amsthm,amsmath}
\usepackage{amssymb}

\newcommand{\de}{\delta}
\newcommand{\e}{\varepsilon}

\newcommand\tc[2]{\theta\left[\begin{matrix}#1\\ #2\end{matrix}\right]}

\newcommand{\CC}{{\mathbb{C}}}

\newcommand{\HH}{{\mathbb{H}}}

\newcommand{\PP}{{\mathbb{P}}}

\newcommand{\ZZ}{{\mathbb{Z}}}

\newcommand{\calH}{{\mathcal H}}

\newcommand{\calI}{{\mathcal I}}
\newcommand{\calA}{{\mathcal A}}

\newcommand{\calM}{{\mathcal M}}

\newcommand{\calU}{{\mathcal U}}
\newcommand{\calS}{{\mathcal S}}

\newcommand{\op}{\operatorname}

\newcommand{\Sp}{\op{Sp}}

\newcommand{\Sym}{\op{Sym}}

\theoremstyle{plain}
\newtheorem{thm}{Theorem}
\newtheorem{lm}[thm]{Lemma}
\newtheorem{prop}[thm]{Proposition}

\theoremstyle{definition}
\newtheorem{df}[thm]{Definition}
\newtheorem{rem}[thm]{Remark}

\begin{document}
\title{The Scorza correspondence in genus 3}
\author{Samuel Grushevsky}
\address{Mathematics Department, Stony Brook University,
Stony Brook, NY 11790-3651, USA. }
\email{sam@math.sunysb.edu}
\thanks{Research of the first author is supported in part by the National Science Foundation under the grant DMS-09-01086.}
\author{Riccardo Salvati Manni}
\address{Dipartimento di Matematica, Universit\`a ``La Sapienza'',
Piazzale A. Moro 2, Roma, I 00185, Italy}
\email{salvati@mat.uniroma1.it}

\begin{abstract}
In this note we prove the genus 3 case of a conjecture of G.~Farkas and A.~Verra on the limit of the Scorza correspondence for curves with a theta-null. Specifically, we show
that the limit of the Scorza correspondence for a hyperelliptic genus 3 curve
$C$ is the union of the curve $\{x,\sigma(x)\mid x\in C\}$ (where $\sigma$ is the
hyperelliptic involution), and twice the diagonal. Our proof uses the
geometry of the subsystem $\Gamma_{00}$ of the linear system $|2\Theta|$,
and Riemann identities for theta constants.

\end{abstract}
\maketitle

\section{Introduction}
Let $\calM_g$ be the moduli space of smooth complex curves of genus $g$,
and let $\calS_g^{\pm}$ be the moduli spaces of smooth spin curves, i.e.~the
moduli of pairs consisting of $C\in\calM_g$ and a line bundle $\eta$ on $C$
such that $\eta^{\otimes 2}=K_C$, and such that $h^0(C,\eta)$ is even
(resp.~odd) on $\calS_g^+$ (resp.~$\calS_g^-$). Given a pair $(C,\eta)\in
\calS_g^+$ such that $h^0(C,\eta)=0$, the {\it Scorza correspondence} is
the curve in $C\times C$ defined by
$$
 S(C,\eta):=\lbrace (x,y)\in C\times C\mid h^0(C,\eta+x-y)> 0\rbrace.
$$
The Scorza correspondence is a classical construction in algebraic geometry, a beautiful modern exposition, many further details, and references for which are given in \cite[section 5.5]{dolgachevbook}.

Since $\eta^{\otimes 2}=K_C$, it follows from Riemann-Roch theorem
that the Scorza correspondence is symmetric, and can be considered in $\Sym^2C$. Moreover, as $\eta$ was assumed
to be non-effective, the correspondence $S$ cannot intersect the diagonal in
$C\times C$.

Considered globally over
$\calS_g^+$, the Scorza correspondences form a codimension one family in $\calS_{g,2}^+$
(the moduli of spin curves with two distinct marked points),
defined away from the locus where $\eta$ is effective. This locus is
called the theta-null  divisor
$$
 \theta_{\rm null}:=\lbrace (C,\eta)\in \calS_g^+\mid h^0(C,\eta)>0\rbrace,
$$
and plays an important role in the study of the birational geometry of the
moduli of spin curves by Farkas and Verra \cite{farkasevenspin,farkasverrainterm}. It is non-empty
for $g\ge 3$.

Naturally one can take the closure of the universal Scorza correspondence
in $\calS_{g,2}^+$, and study the fibers over the theta-null divisor, i.e.~the
strict transform of the Scorza correspondence as the curve acquires a theta-null.
While in general not much is known about the geometry of curves with a theta-null, a test case is provided by hyperelliptic curves (which are in fact characterized by the vanishing of a certain collection of theta-nulls \cite{mumfordbooktheta2,poor}).

The following conjecture, due to Farkas and Verra, was communicated to us by Gavril Farkas:

\smallskip
\noindent {\bf Conjecture.}
{\it For a  generic hyperelliptic curve $C$ with the hyperelliptic involution $\sigma$,
and a vanishing even theta characteristic $\eta$ on $C$, the limit of the
Scorza correspondence in $C\times C$ scheme-theoretically is the union of
the curve $\lbrace x\times\sigma(x)\mid x\in C\rbrace$, and of the
diagonal with multiplicity two.}

Our main result, theorem \ref{mainthm}, is the proof of this conjecture for the case of $g=3$, for all hyperelliptic  curves of genus $3$ --- in which case each hyperelliptic curve has precisely one vanishing theta-null.

\smallskip
Another motivation for studying the Scorza correspondence is
the problem of constructing complete subvarieties of $\calM_g$. Recall that by
a theorem of Diaz \cite{diaz} (see also \cite{grkr} for a different
new approach) there do not exist complete subvarieties of $\calM_g$ of dimension
larger than $g-2$, while all known explicit examples of complete subvarieties
are of dimension equal to a constant multiple of $\ln g$ (for $g\gg 0$)
(see \cite{hamobook} for more details).

A common theme in constructing such complete subvarieties is by using
an appropriate cover construction. The starting point for such constructions
is the complete curve
$X\subset\calM_3$ constructed explicitly by Zaal \cite{zaal} (the existence of such
a curve follows a priori from the existence of the projective Satake compactification
$\calM_3^{Sat}$ such that the boundary $\partial\calM_3^{Sat}$ is codimension two,
so that cutting $\calM_3^{Sat}\subset\PP^N$ by sufficiently general
hypersurfaces yields a curve that can be made to avoid $\partial\calM_3^{Sat}$ ---
and which is thus contained in $\calM_3$). By considering
a suitable cover of $X$ one thus obtains a complete surface in $\calM_7$, and this
can be iterated further.

Alternatively, one could try to use the Scorza
correspondence (this idea is due to Gavril Farkas): indeed, consider the preimage of $X$ in
$\calS_3^+$, and consider the universal Scorza correspondence $S$ over it. For
any $(C,p,q)\in S\subset\calS_{3,2}^+$ if $p$ and $q$ are distinct, the double cover
of $C$ branching at $p$ and $q$ is a smooth curve of genus $6$. Thus if the Scorza
correspondence were to never intersect the diagonal of $C\times C$ (including over
the hyperelliptic locus), by varying $C\in X$ and $p,q\in S$ one would get
a complete surface in $\calM_6$. Thus our main result shows that this cannot be made
to work, as indeed the complement of the hyperelliptic locus $\calM_3\setminus\calH_3$
is affine  and does not contain any complete curves. Note that in characteristic $p>2$
a complete surface in $\calM_6$ was constructed by Zaal in \cite{zaal2}.

\smallskip
The method of our proof is by using the difference map $C\times C\to C-C\subset J(C)$,
and using the explicit description of the image as the base locus of the
linear subsystem $\Gamma_{00}\subset|2\Theta|$ on the Jacobian. Using
the explicit knowledge of this linear system and suitable identities for theta
constants then yield the defining equations for the strict transform of the
Scorza correspondence over the hyperelliptic curves, which allows us to prove the main result.

\bigskip
{\it Acknowledgements:} We would like to thank Gavril Farkas for bringing the problem to
our attention and for insightful conversations about the geometry of spin curves. We are grateful to Igor Dolgachev for interesting discussions about the classical geometry of curves of genus 3.

\section{Theta constants as modular forms}
In this section we gather some definitions and results about theta constants and
the rings of modular forms.

We denote by $\HH_g$ the Siegel upper half-space consisting of $g\times g$ complex
symmetric matrices $\tau$ with positive-definite imaginary part.
For $\tau\in\HH_g$ we denote by $A_\tau:=\CC^g/\ZZ^g\tau+\ZZ^g$ the abelian
variety corresponding to $\tau$. We denote by
$$
 \theta(\tau,z):=\sum\limits_{n\in\ZZ^g} \exp(\pi in^t\tau n+2\pi i n^tz)
$$
the Riemann theta function $\theta:\HH_g\times\CC^g\to\CC$.
The zero locus in $z$ of the theta function is invariant under translating $z$ by
$\ZZ^g\tau+\ZZ^g$, and thus defines the theta divisor $\Theta_\tau\subset A_\tau$,
which gives a principal polarization on $A_\tau$.

For $\e,\de\in(\ZZ/2\ZZ)^{2g}$ (which we think of as a two-torsion point $m=(\e\tau+\de)/2$on $A_\tau$) we denote by $\theta_m(\tau,z):=\tc\e\de(\tau,z)$ the theta
function with characteristics --- which, up to an easy exponential factor,
is the same as $\theta(\tau,z+m)$. It is known that $h^0(A_\tau,2\Theta_\tau)=2^g$ for any $\tau\in\HH_g$ and
that the basis for sections is given by theta functions of the second
order: for $\e\in(\ZZ/2\ZZ)^g$ we denote
$$
\Theta[\e](\tau,z):=\tc{\e}{0}(2\tau,2z)=\sum\limits_{n\in\ZZ^g} \exp\left(2\pi i
\left(n+\e/2\right)^t\left(\tau(n+\e/2)+2z\right)\right)
$$

Since for any two-torsion point $m\in A_\tau[2]$ the square $\theta_m^2(\tau,z)$
is a section of $2\Theta_\tau$ (the shift $2m$ is zero), it is a linear combination
of theta functions of the second order, given explicitly by Riemann's bilinear
addition formula:
\begin{equation}\label{bilinear}
 \tc\e\de(\tau,z)^2=\sum\limits_{\sigma\in(\ZZ/2\ZZ)^g} (-1)^{\sigma\cdot\de}
 \Theta[\sigma](\tau,z)\Theta[\sigma+\e](\tau,0).
\end{equation}

\smallskip
The symplectic group $\Sp(2g,\ZZ)$ acts on $\HH_g$ by
$$
 \begin{pmatrix} A&B\\ C&D \end{pmatrix}\circ\tau=(A\tau+B)(C\tau+D)^{-1},
$$
where we represent an element $\gamma\in\Sp(2g,\ZZ)$ as four $g\times g$ blocks.
The quotient under this action is the moduli space of principally polarized
abelian varieties (ppav) $\calA_g=\HH_g/\Sp(2g,\ZZ)$. For future use we also
denote $\calU_g\to\calA_g$ the universal family of ppav over $\calA_g$ (which
is a quotient of $\HH_g\times\CC^g$ under a suitable action of
$\Sp(2g,\ZZ)\rtimes\ZZ^{2g}$).

\begin{df}
For a finite index subgroup $\Gamma\subset\Sp (2g,\ZZ)$ and an integer $r$,  a multiplier
system of weight $r/2$ is a map $v:\Gamma\to \CC^*$, such that the
map
$$
  \gamma\mapsto v(\gamma)\det(C\tau+D)^{r/2}$$
satisfies the cocycle condition for every $\sigma\in\Gamma$ and
$\tau\in\HH_g$ (note that the function $\det(C\tau+D)$ possesses a
square root).
\end{df}
Clearly a multiplier system of integral weight is a
character.

\begin{df} A function $f:\HH_g\to \CC$ is called a modular form of weight ${r/2}$
with multiplier $v$, with respect to a finite index subgroup $\Gamma\subset\Sp(2g,\ZZ)$ if
\begin{equation}\label{transform}
  f(\gamma\cdot\tau)=v(\gamma)\det(C\tau+D)^{r/2} f(\tau)\qquad\qquad\forall
  \tau\in\HH_g,\forall\gamma\in\Gamma,
\end{equation}
and if additionally $f$ is holomorphic at all cusps of $\HH_g/\Gamma$.
\end{df}
We denote by $[\Gamma, r/2, v]$ the finite dimensional space of modular forms.
We denote $L$ the bundle of modular forms of weight 1 on $\calA_g$, so that
by definition a modular form of weight $r$ (with trivial multiplier) with respect
to the entire group $\Sp(2g,\ZZ)$ is a section of $L^{\otimes r}$: so we have
$[\Sp(2g,\ZZ),r,1]=H^0(\calA_g,L^{\otimes r})$.

The restrictions of theta functions with characteristics (resp.~theta functions
of the second order) to $z=0$ are called theta constants with characteristics
(resp.~of the second order). It is known that $\theta_m(\tau,0)$ are modular
forms of weight one half with respect to a certain finite index normal subgroup
$\Gamma_g(4,8)\subset\Sp(2g,\ZZ)$ (the action of $\Gamma_g(2)$ adds certain
fourth roots of unity, while the action of full $\Sp(2g,\ZZ)$ adds certain
complicated eighth roots of unity and
permutes characteristics under an affine action), while theta constants of
the second order are modular forms, also of weight one half, with respect
to a bigger subgroup $\Gamma_g(2,4)$. Using the same method as in \cite{top},
we can prove that the commutator of the group  $\Gamma_g(2,4)$ is the group
$\Gamma_g(2,4, 8)$.  Moreover its characters appear in the transformation formula of
theta constants, see \cite{smlevel2}.
We denote $\calA_g(2,4):=\HH_g/\Gamma_g(2,4)$ the finite Galois cover of
$\calA_g$ corresponding to the level $(2,4)$ subgroup. Theta constants of the
second order are known to define a morphism
$$
  Th:\calA_g(2,4)\to\PP^{2^g-1};\qquad \tau\mapsto\lbrace\Theta[\e](\tau,
  0)\rbrace_{\e\in(\ZZ/2\ZZ)^g},
$$
that is generically 1-to-1, finite, and is
known to be an embedding  when  $g\leq 3$ --- see \cite{runge1,smsubrings}
(the similar map of $\calA_g(4,8)$ given by theta constants with characteristics is
known to be an embedding for all $g$). Letting $v$  be the  multiplier of the theta constants of
the second order, we  set
$$
  M(\Gamma_g(2,4) , v):=\bigoplus_{r=0}^{\infty} [\Gamma_g(2,4), r/2, v^r].
$$
This is an integrally closed ring; moreover, for $g\leq 3$
$$
  M(\Gamma_g(2,4) , v)=\CC\left[\Theta[\sigma](\tau,0)\right],
$$
i.e.~theta constants of the second order generate this ring of modular forms.
We refer to \cite{igusabook,runge1,smlevel2} for more details on theta
constants and associated maps.

\section{The linear system $\Gamma_{00}$}
In this section we recall the definition of the linear system $\Gamma_{00}$
introduced by van Geemen and van der Geer \cite{vgvdg}, and some results about it.

For a curve $C\in\calM_g$ let $A:C\hookrightarrow J(C)$ be the Abel-Jacobi
embedding of the curve into its Jacobian. Let $C-C\subset J(C)$ denote the
(singular) surface that is the image of the map $C\times C\to J(C)$ given
by $(p,q)\mapsto A(p)-A(q)$ (which blows down the diagonal). Recall
that the degree of the map $C\times C\to C-C$ is equal to 1 for
non-hyperelliptic curves, and 2 for hyperelliptic curves.

For any ppav $A_\tau$ the linear subsystem $\Gamma_{00}\subset|2\Theta_\tau|$ is
defined by van Geemen
and van der Geer  \cite{vgvdg} to consist of those sections vanishing at the
origin $z=0$ to order at least 4. For a Jacobian, $\Gamma_{00}$ coincides with
the vector subspace of $H^0(A_\tau,2\Theta_\tau)$ consisting of sections
vanishing along $C-C$, see \cite{vgvdg}. Welters \cite{weltersC-C} showed that
as a set the base locus of $\Gamma_{00}$ on a Jacobian is equal to $C-C$
(for $g=4$ together with two exceptional points corresponding to the
$g_1^3$'s on the curve). Izadi \cite{izadiC-C} showed that the base locus
of $\Gamma_{00}$ on a non-hyperelliptic Jacobian is equal to $C-C$ as a scheme.

For any indecomposable ppav $A$ the dimension of $\Gamma_{00}$ is equal
to $2^g-g(g+1)/2-1$, and a set of generators for $\Gamma_{00}$ is constructed in \cite{vgvdg} in the
following way (see \cite{grsmtwopoint} for a different approach).
Let $I\in I(\overline{Th(\calA_g(2,4))})$ be a homogenous polynomial in theta
constants of the second order vanishing identically. van Geemen and van der Geer \cite{vgvdg}
show (using the heat equation for the theta function) that the
expression
\begin{equation}\label{F}
  F_I:=\sum\limits_{\e\in(\ZZ/2\ZZ)^g} \frac{\partial I}{\partial\Theta[\e]
  (\tau,0)}\Theta[\e](\tau,z)
\end{equation}
then gives an element in $\Gamma_{00}$, and moreover they prove that such $F_I$
for all $I$ in the ideal generate the linear system $\Gamma_{00}$ for any
indecomposable ppav.


\section{The surface $C-C$ in genus 3}
In this section we specialize the preceding discussion to write down in genus 3
the defining equation for the universal family of surfaces $C-C\subset\calU_3$.

In genus 3 the image $\overline{Th(\calA_3(2,4))}$ is a hypersurface in $\PP^7$.
The equation for this hypersurface, of degree 16, was
first obtained by Schottky (perhaps Frobenius), and we quote the formula from
\cite{vgvdg}. For the degree four polynomials in theta constants defined as
$$
 r_1:=\tc{0&0&0}{0&0&0}\tc{0&0&0}{1&0&0}\tc{0&0&0}{0&1&0}\tc{0&0&0}{1&1&0}
$$
$$
 r_2:=\tc{0&0&1}{0&0&0}\tc{0&0&1}{1&0&0}\tc{0&0&1}{0&1&0}\tc{0&0&1}{1&1&0}
$$
$$
 r_3:=\tc{0&0&0}{0&0&1}\tc{0&0&0}{1&0&1}\tc{0&0&0}{0&1&1}\tc{0&0&0}{1&1&1}
$$
the Riemann quartic addition theorem gives $r_1-r_2-r_3=0$, which implies
\cite[formula (7)]{vgvdg}
\begin{equation}\label{I3}
 I_3(\tau):=r_1^2+r^2_2+r_3^2-2r_1^2r_2^r-2r_1^2r_3^2-2r_2^2r_3^2=0.
\end{equation}

Using Riemann's bilinear addition formula (\ref{bilinear}) one can rewrite $I_3$ as a
polynomial of degree 16 in theta constants of the second order
$\Theta[\sigma](\tau)$. It then turns out that this polynomial on $\PP^7$ is
in fact independent of the choice of a syzygetic subspace used to define $r_i$,
and is the one defining equation for $\overline{Th(\calA_3(2,4))}\subset\PP^7$.

An equivalent expression for $I_3(\tau)$, manifestly invariant under the action
of $\Sp(2g,\ZZ)$, was obtained by
Igusa \cite{igusachristoffel}, who showed that up to a constant factor
$$
 I_3(\tau)=2^3\sum\limits_{\e,\de\in(\ZZ/2\ZZ)^g}\tc\e\de^{16}(\tau,0)-\left(
 \sum\limits_{\e,\de\in(\ZZ/2\ZZ)^g}\tc\e\de^{8}(\tau,0)\right)^2.
$$

Thus from the results of van Geemen and van der Geer it follows that for
any indecomposable ppav $A_\tau$ of dimension 3 the one-dimensional linear
system $\Gamma_{00}$ is generated by the function
$$
 F(\tau,z):=F_{I_3}(\tau,z).
$$
To write $F$ explicitly, note from \cite{vgvdg} that instead of taking
partial derivatives in (\ref{F}) it is possible to take partial
derivatives with respect to squares of theta constants with
characteristics, and then sum over all even characteristics, i.e.~up to
a constant factor we have
\begin{equation}\label{F2}
 F(\tau,z)=\sum\limits_{\e,\de\in(\ZZ/2\ZZ)^6}\frac{\partial I_3}{\partial
 \tc\e\de(\tau,0)^2}\tc\e\de(\tau,z)^2
\end{equation}
and using formula (\ref{I3}) this can be written out explicitly.

\section{The Scorza correspondence in genus 3}
In this section we describe analytically the (universal) Scorza correspondence
in genus 3.

Recall that the Scorza correspondence for $(C,\eta)\in\calS_g^+$ for $\eta$
non-effective is the set of pairs of points $(x,y)\in C\times C$ such
that $\eta+x-y$ is an effective divisor. The choice of a spin structure
$\eta$ fixes a choice of a symmetric principal polarization $\Theta_\eta$ on the
Jacobian $J(C)$ such that $\Theta_\eta|_{A(C)}=\eta$. The
image of the Scorza correspondence under the map $C\times C\longrightarrow C-C\subset
J(C)$ is then the set of points $x-y\in C-C$ such that $\theta_\eta(\tau,x-y)=0$,
where $\tau$ is the period matrix of $C$. Thus the image of the Scorza
correspondence in $J(C)$ is the intersection of
$C-C$ with the theta divisor. Using the above
description of $C-C$ as the zero locus of $F(\tau,z)$ analytically
we can write down the Scorza curve in $A_\tau=J(C)$ for $C$
non-hyperelliptic as
\begin{equation}\label{scorza}
 S_\tau:=(C-C)\cap\Theta_\tau=\lbrace z\in A_\tau\mid F(\tau,z)=
 \theta_\eta(\tau,z)=0\rbrace
\end{equation}
In what follows, working analytically over the Siegel space, to simplify notations
we will choose the theta characteristic to be zero (since the deck transformations
of the cover $\calS_g^+\to\calM_g$ act transitively on the set of $\eta$, this
will be enough).

Note that we can define the Scorza variety $S_\tau\subset A_\tau$ by the same
formula for any $\tau\in\HH_3$, even if the corresponding curve is reducible or
hyperelliptic. We then see that $S_\tau$ is a curve unless the
irreducible divisor $\Theta_\tau\subset A_\tau$ coincides with the surface
$C-C$, which we know is given in $A_\tau$ by the equation $\lbrace
F(\tau,z)=0\rbrace$. Since $0\in C-C$, for
this to happen we must have in particular
$0\in\Theta_\tau$, i.e.~we must have $\theta(\tau,0)=0$. Thus this
happens precisely on the hyperelliptic locus (with the
appropriate choice of the theta characteristic, i.e.~on $\theta_{\rm null}\subset
\calS_3^+$). In this case
$\theta(\tau,z)$ is an even function of $z$ vanishing at 0, so it
vanishes to the second order at $z=0$, its square vanishes to the fourth order, and thus $\theta^2(\tau,z)$ is the
non-zero generator of $\Gamma_{00}$. We thus proved the following

\begin{lm}
On the hyperelliptic locus given by $\theta(\tau,0)=0$ the function $F$ vanishes
on the theta divisor: this is to say that we have
$$
 \theta(\tau,0)=\theta(\tau,z)=0\Longrightarrow F(\tau,z)=0.
$$
\end{lm}

\section{ An expression for $F(\tau, z)$}
We now consider the embedding of the universal Kummer variety
of dimension 3:
$$
  UTh:\calU_3(2,4)/\pm 1\hookrightarrow \PP^7\times\PP^7
$$
given by
$$
 (\tau,z)\mapsto\lbrace\Theta[\e](\tau,
 0)\rbrace_{\e\in(\ZZ/2\ZZ)^g}\times \lbrace\Theta[\e](\tau,
 z)\rbrace_{\e\in(\ZZ/2\ZZ)^g}
$$

The defining ideal $\calI$ for the image of $UTh$ is known quite explicitly.
Indeed, let $p_1$ and $p_2$ be the two projections from $\PP^7
\times\PP^7$ onto the factors. Then the map $p_1\circ UTh$ factors
through $\calA_3(2,4)$, where it is simply $Th$, and thus $I_3$
is the single defining equation for the closure of $p_1\circ UTh(\calU_3(2,4)$
in $\PP^7$.

On the other hand, for a fixed abelian variety (which, since $Th
=p_1\circ UTh$ is an embedding, means for $p_1\circ UTh$ fixed),
the map $p_2\circ UTh$ is the map $|2\Theta_\tau|:A_\tau/\pm
\hookrightarrow \PP^7$, equations for which are known by work
of Mumford \cite{mumfordequations1,mumfordequations3,mumfordequations3}.

In fact these equations are polynomial equations for $\Theta[\sigma](\tau,z)$
with coefficients themselves being polynomials in $\Theta[\e](\tau,0)$ ---
thus they are explicit polynomials on $\PP^7\times\PP^7$.

The above lemma then yields the following
\begin{prop}
The function $F(\tau,z)$ considered as a bihomogenous polynomial (of bidegree
$(15,1)$) on $\PP^7\times\PP^7$, lies in the radical of the ideal
generated by $\calI$ and the functions $\theta(\tau,0)$ and $\theta^2(\tau,z)$.
\end{prop}
This proposition is of course simply a reformulation of the lemma; note
that while by Riemann's bilinear addition formula $\theta^2(\tau,0)$ is a polynomial
in theta constants of the second order, $\theta(\tau,0)$ itself is not.

We will now determine explicitly an expression for $F$ in terms of the
generators of this ideal, using the explicit formula
(\ref{F2}) for $F$.
Let $V_i\subset(\ZZ/2\ZZ)^6$ for $i=1,2,3$ be the set of four theta
characteristics appearing in the product $r_i$. We need to compute the partial
derivatives of (\ref{I3}) with respect to various theta constants.
Obviously for $m\not\in (V_1\sqcup V_2\sqcup V_3)$ such a partial derivative is
zero, while for $m\in V_i$ we have
$$
\frac{\partial I_3}{\partial \theta_m^2(\tau,0)}=2\left(2r_i^2-
 \sum\limits_{j=1}^3r_j^2\right)\frac{r_i^2}{\theta_m^2(\tau,0)}.
$$

Summing over all $m$ yields
$$F(\tau,z)= \sum_{i=1}^3 2\left(2r_i^2-
 \sum\limits_{j=1}^3r_j^2\right)\sum_{m\in V_i}\frac{r_i^2}{\theta_m^2(\tau,0)}\theta_m^2(\tau,z) $$

We want to evaluate $F(\tau,z)$ on the hyperelliptic locus, which is to say on
$\theta_{\rm null}\subset\calS_3^+$. Lifting to $\HH_3$ it means we are working on the
locus of $\tau$ such that $\theta(\tau,0)=\tc{0&0&0}{0&0&0}(\tau,0)=0$.  In this case of
course we have $r_1=0$, which implies $r_2+r_3=0$ and thus $r_2^2-r_3^2=(r_2-r_3)(r_2+r_3)=0$.
Since this product vanishes on the locus $\theta(\tau,0)$, it implies  that
the modular form $r_2^2-r_3^2$ is divisible
by $\theta(\tau,0)$.

\begin{rem}\label{notdivisible}
For future use we note that $r_2^2-r_3^2$ is not divisible
by $\theta(\tau,0)^2$: indeed, we know that $r_2+r_3=r_1$ is divisible by $\theta(\tau,0)$.
If $r_2-r_3$ were also divisible by $r_1$, this would imply that $r_2$ and $r_3$ were both
divisible by $r_1$, which is of course false as zero loci of different theta constants are different.
On the other hand, if $r_2+r_3=r_1$ were
divisible by $\theta(\tau,0)^2$, then $r_1$ would be divisible by
$\theta(\tau,0)^2$, which is also not the case.
\end{rem}

This computation allows us to write $F$ explicitly. To simplify the formulas, we introduce notation
for various expressions appearing in the formulas for the derivatives of $I$ with respect to
various $\theta_m^2(\tau,0)$. To this end, for any $m\in V_i,\, i=2,3$ we set
$$
 A_m(\tau):=(-1)^i 2(r_2-r_3)\frac{r_1}{\theta(\tau,0)}\cdot\frac{r_i^2}{\theta_m^2(\tau,0)};
$$
for any  $m\in V_1\setminus0$ we set
$$
  B_m(\tau):= 2(r_1^2- r_2^2- r_3^2) \frac{r_1^2}{\theta^2(\tau,0)\theta_m^2(\tau,0)},
$$
and finally we set
$$
  C(\tau):= 2(r_1^2- r_2^2- r_3^2) \frac{r_1^2}{\theta^2(\tau,0)}.
$$
In this notation we then have by adding all the partial derivatives of $I$
\begin{equation}\label{fr}
\begin{aligned}
  F(\tau,z)&=\theta(\tau,0)\sum\limits_{m\in V_2\sqcup V_3}A_m(\tau)\theta_m^2(\tau,z)\\
  &+\theta^2(\tau,0)\sum\limits_{m\in V_1\setminus 0} B_m(\tau)\theta_m^2(\tau,z)+C(\tau)\theta^2(\tau,z)
\end{aligned}
\end{equation}
with $ A_m(\tau)$ not divisible by  $\theta(\tau,0)$. Denoting the first two sums $A(\tau,z)$ and $B(\tau,z)$ respectively, we can write this as
\begin{equation}\label{frsimple}
  F(\tau,z)=\theta(\tau,0)A(\tau,z)+\theta^2(\tau,0)B(\tau,z)+C(\tau)\theta^2(\tau,z).
\end{equation}
Recall that $F(\tau,z)=0$ within $A_\tau=J(C)$ is the defining equation for the surface $C-C\subset J(C)$, which in particular
contains zero, so that we must have
$$
 0=\theta(\tau,0)A(\tau,0)+\theta^2(\tau,0)(B(\tau,0)+C(\tau))
$$
for all $\tau$, which implies that
$$
 A(\tau,0)=-\theta(\tau,0)(B(\tau,0)+C(\tau))
$$
for all $\tau$, and in particular we get
\begin{lm}\label{Avanishes}
The expression $A(\tau,0)$ defined above vanishes on the component of the hyperelliptic
locus where $\theta(\tau,0)=0$: we have $\theta(\tau,0)=0\Longrightarrow A(\tau,0)=0$.
\end{lm}

\begin{rem}
Note that we could have formulas (\ref{fr}) and (\ref{frsimple}) satisfied with other $ A_m(\tau),
B_m(\tau), C(\tau)$, as the theta functions $\theta_m^2(\tau,z)$ appearing there are not linearly
independent. However, the conclusion above about the vanishing of $A(\tau,0)$ on the hyperelliptic
locus is independent of this assumption.
\end{rem}

\section{The Scorza correspondence over the hyperelliptic locus in genus 3}
Recall that the Scorza correspondence within $A_\tau$ is given by the equations
$$
S_\tau=\lbrace\theta(\tau,z)=F(\tau,z)=0\rbrace.
$$
Substituting here the expression for $F(\tau,z)$ from formula (\ref{frsimple}) and
noticing that the term with $C(\tau)$ vanishes on the theta divisor we get equivalently
$$
 S_\tau=\lbrace 0=\theta(\tau,z)=\theta(\tau,0)(A(\tau,z)+\theta(\tau,0)B(\tau,z)\rbrace.
$$
Away from the component of the hyperelliptic locus where $\theta(\tau,0)=0$ we can thus take out that
factor from the second equation. Thus the limit of the Scorza correspondence (i.e.~the intersection of the closure of the
family of $S_\tau\subset \calU_3$ with the fiber) as $\tau$ approaches the period
matrix $\tau_0$ of a hyperelliptic curve $C_0\in\calH_3$ such that $\theta(\tau_0,0)=0$  is given by
\begin{equation}\label{Stau0}
 S_{\tau_0}=\lbrace 0=\theta(\tau_0,z)=A(\tau_0,z)\rbrace.
\end{equation}
We note that if $A(\tau,z)$ were divisible by $\theta(\tau,0)$, then in the defining equations
for $S_\tau$ we could take out the factor of $\theta^2(\tau,0)$, and thus the formula for $S_{\tau_0}$ would be different:
it would involve $A(\tau,z)/\theta(\tau,0)$ instead of $A(\tau,0)$.

We will now prove that this is not the case:
\begin{prop}
The expression $A(\tau,z)$ is not divisible by $\theta(\tau,0)$: there does not
exist a holomorphic function $D(\tau,z)$ --- which would then be a modular function ---
such that $A(\tau,z)=D(\tau,z)\theta(\tau,0)$. Equivalently, there does not exist a
holomorphic $D$ such that
\begin{equation}\label{frm}
  F(\tau,z)=  \theta^2(\tau,0)D(\tau,z)+C(\tau)\theta^2(\tau,z)
\end{equation}
\end{prop}

\begin{proof}
Suppose for contradiction that we have (\ref{frm}) valid with some
holomorphic function $D(\tau,z)$. Since $C(\tau)\in A(\Gamma_3(2,4), v)$, it
is a polynomial in the theta constants of the second order. Furthermore,
$D(\tau,z)$ for fixed $\tau$ as a function of $z$ is a section of $2\Theta_\tau$
(as are the other terms in (\ref{frm}). Since theta functions of the
second order form a basis of $H^0(A_\tau,2\Theta_\tau)$, we must then have
$$
  D(\tau,z)=\sum_{\e\in(\ZZ/2\ZZ) ^3}D_{\e}(\tau)\Theta[\e](\tau,z)
$$
with $D_{\e}(\tau)$ some modular forms of weight $\frac{13}{2}$ wrt $\Gamma_3(2,4)$,
and hence each $D_\e$ a polynomial of degree 13 in the theta constants of the second order.

Using Riemann's bilinear addition theorem for $\theta^2(\tau,z)$ we would then get from
(\ref{frm}) the identity
$$
  F(\tau,z)= \sum_{\e\in(\ZZ/2\ZZ) ^3} (\theta^2(\tau,0)D_{\e}(\tau)+
  C(\tau)\Theta[\e](\tau,0))\Theta[\e](\tau,z)
$$
$$
  =\sum_{\e\in(\ZZ/2\ZZ)^3}F_{\e}(\tau)\Theta[\e](\tau,z)
$$
with  $F_{\e}(\tau)$ a polynomial of degree 15 in the $\Theta[\e](\tau,0)$.

Using Maple  we verified that in the polynomial ring generated by theta functions
and constants of the second
order, by $\Theta[\e](\tau,z)$ and $\Theta[\e](\tau,0)$, the polynomial $F(\tau,z)$
does not lie in the ideal generated by $\theta^2(\tau,0)$ and $\theta^2(\tau,z)$,
and thus the above expression is impossible.
\end{proof}

We will now investigate the resulting limit of the Scorza correspondence, given by formula
(\ref{Stau0}). We compute by definition
$$
 A(\tau_0,z)=\frac{2r_1(\tau_0)(r_2(\tau_0)-r_3(\tau_0))}{\theta(\tau_0,0)}\left(
 r_2^2\sum\limits_{m\in V_2} \frac{\theta_m^2(\tau_0,z)}{\theta_m^2(\tau_0,0)}-r_3^2\sum\limits_{m\in V_3} \frac{\theta_m^2(\tau_0,z)}{\theta_m^2(\tau_0,0)}\right),
$$
which using $0=r_1(\tau_0)=r_2(\tau_0)+r_3(\tau_0)$ implies that its vanishing locus in $z$ is equal to that of
$$
 H(\tau_0,z):=r_2^2(\tau_0)\sum\limits_{m\in V_2} \frac{\theta_m^2(\tau_0,z)}{\theta_m^2(\tau_0,0)}-r_3^2(\tau_0)\sum\limits_{m\in V_3} \frac{\theta_m^2(\tau_0,z)}{\theta_m^2(\tau_0,0)}.
$$
Recall that by a conjecture of Farkas and Verra we expect to have $S_{\tau_0}=\lbrace
x-\sigma(x)|x\in C_0\rbrace$, where $\sigma$ denotes the hyperelliptic involution on $C_0$. We
first prove the inclusion.

\begin{prop}
The function $H(\tau_0,z)$ vanishes identically for $z=x-\sigma(x)$, and thus $S_{\tau_0}\supset\lbrace x-\sigma(x)|x\in C_0\rbrace$.
\end{prop}
\begin{proof}
We choose
an Abel-Jacobi embedding of $C_0$ into its Jacobian in such a way that the involution lifts to
$-1$, so that we need to check that $H(\tau_0,2x)=0$ for any $x$ on the curve.
As a function of $z$, $H(\tau,z)$ is a section of $2\Theta_\tau$; hence $H(\tau,2z)$ is a section of $8\Theta_\tau$. Since the theta function has degree 3 on a
curve of genus 3, $H(\tau_0,2x)$ is then a section of a line bundle of degree 24 on
$C_0$. Thus if we show it has more than 24 zeroes (counted with multiplicity), it is identically zero.

Note that by definition the 8 Weierstrass points are the fixed points of the involution
$\sigma$, so if $x$ is a Weierstrass point, we have $H(\tau_0,x-\sigma(x))=H(\tau_0,0)=0$.

Since we are in genus 3 we can choose the Abel-Jacobi map of $C_0$ in such a way that the eight Weierstrass point map to some points $x_0=0, x_1, \dots x_7$ of the Jacobian $J(C_0)$, and the associated characteristics are $0$ and $m_1, \dots , m_7$ with all $m_i$ odd.  A consequence of this   choice is $\theta(\tau_0,0)=0$.  So our choice of Abel-Jacobi map is consistent with our hypothesis.

Since $H(\tau_0,z)$ is an even function of $z$, its gradient vanishes at each two-torsion point on $J(C_0)$, and thus automatically $H$ vanishes at each Weierstrass
point to order at least 2. So we get 16 zeroes, counted with multiplicity. We now show that $H$ vanishes at each Weierstrass point to order at least 4, so the total number of zeroes is at least 32 (and thus $H(\tau_0,2x)$ will vanish identically).
To achieve this we will check that the second derivative is zero (the third derivative is zero by parity).

Indeed, we need to check that $\partial_U\partial_U H(\tau_0,z)|_{z=0}=0$, where $U$ denotes the tangent vector to the curve $C_0\subset J(C_0)$ at the Weierstrass point (which, without loss of generality, we can take to be 0).
Note that we have $\theta(\tau_0,0)=0$, so that $r_1(\tau_0)=0$ and thus $r_2(\tau_0)+r_3(\tau_0)=0$. By the heat equation, and using the parity
of each $\theta_m$ as a function of $z$, we then get, up to a constant factor,
$$
 \partial_U\partial_U H(\tau_0,z)|_{z=0}=
 2r_2^2(\tau_0)\sum\limits_{m\in V_2} \frac{\partial_{U\otimes U}\theta_m(\tau_0,0)}{\theta_m(\tau_0,0)}-2r_3^2(\tau_0)
 \sum\limits_{m\in V_3} \frac{\partial_{U\otimes U}\theta_m(\tau_0,0)}{\theta_m(\tau_0,0)}
$$
$$
 =\partial_{U\otimes U}(r_2^2(\tau_0)-r_3^2(\tau_0))=2r_2(\tau_0)\partial_{U\otimes U}
 (r_2(\tau_0)+r_3(\tau_0)
$$
where $\partial_{U\otimes U}$ denotes the derivative in the $\tau$ direction with
respect to the rank one matrix $U\otimes U$.
To compute this derivative, we will take the derivatives of  $r_2(\tau)+r_3(\tau) -r_1(\tau) \equiv 0$.  Since this is identically zero, also the derivative is  zero, hence
$$
 \partial_{U\otimes U}r_2(\tau_0)+ \partial_{U\otimes U}r_3(\tau_0)= \partial_{U\otimes U} r_1(\tau_0)
 =\frac{r_1(\tau_0)}{\theta(\tau_0)}\partial_{U\otimes U}\theta(\tau_0).
$$
We now note   that the theta function vanishes identically on $C_0$ since
$$\Theta_{\tau_0}=C_0-C_0\supset C_0-x_0= C_0-0= C_0.$$
So we have $\theta(\tau_0,x)=0$ for any point $x\in C_0$.
Using the heat equation again, we then see that
$$
\partial_{U\otimes U}\theta(\tau_0)=\partial_U\partial_U\theta(\tau_0,z)|_{z=0}=0
$$
is zero as the derivative of the zero function (recall that $U$ is the tangent vector to $C_0\subset J(C_0)$ at zero).
\end{proof}

In fact this statement implies the conjectural description of the strict transform of the correspondence on $C_0\times C_0$ in genus 3:

\begin{thm}\label{mainthm}
The conjecture of Farkas and Verra holds in genus 3, i.e.~the limit of the
Scorza correspondence within $C_0\times C_0$, where $C_0$ is a hyperelliptic curve
of genus 3 with involution $\sigma$, is the union of the curve $\lbrace (x,\sigma(x))\mid x\in C_0\rbrace$ and the diagonal with multiplicity 2.
\end{thm}
\begin{proof}
Indeed, note that the image of the locus $\lbrace (x,\sigma(x))\mid x\in C_0\rbrace$ in
the Jacobian $J(C_0)$, being equal to $S_{\tau_0}=\lbrace 2x\mid x\in C_0\rbrace\subset
C_0-C_0$, is a curve that passes through the singular point $0\in C_0-C_0$ with
multiplicity 2. Thus its preimage on $C_0\times C_0$ contains the two components as claimed, and
it remains to show that there are no extra components.

This can be done by an intersection number computation. Indeed, for a very general curve $C$ (and for a very general hyperelliptic curve $C_0$) the Neron-Severi group $NS(\Sym^2 C)$ is generated by two classes: the fiber $f:=\lbrace (x,p)\mid \forall x\in C\rbrace$ and the diagonal $\delta:=\lbrace (x,x)\mid \forall x\in C\rbrace$. For a non-hyperelliptic curve $C$ the Scorza correspondence is the locus $s:=\lbrace (x,y)\mid \theta(x-y)=0\rbrace$. It has intersection numbers $s\cdot \delta=0$ (since the intersection is empty) and $s\cdot f=3$ (since the degree of the theta function on the curve is equal to $g$, which is 3 for our case). Now for a hyperelliptic curve $C_0$ on $\Sym^2C_0$ we let $h:=\lbrace (x,\sigma(x))\mid \forall x\in C_0\rbrace$ and compute
$h\cdot \delta=8$ (since these are the fixed points of the involution) and $h\cdot f=1$. Since we also have $\delta^2=2-2\cdot 3=-4$ and $\delta\cdot f=1$, we get
$$
 (h+2\delta)\cdot\delta=8-2\cdot 4=0\quad{\rm and}\quad (h+2\delta)\cdot f=1+2\cdot 1=3.
$$
It thus follows that if the strict transform of $s$ were to contain any curve besides $h+2\delta$, the intersection of this curve with both $\delta$ and $f$ would be zero. Thus the limit of the Scorza correspondence over a hyperelliptic curve contains no further components, and is as claimed.
\end{proof}


\begin{thebibliography}{vGvdG86}

\bibitem[Dia84]{diaz}
S.~Diaz.
\newblock A bound on the dimensions of complete subvarieties of {${\mathcal
  M}_{g}$}.
\newblock {\em Duke Math. J.}, 51(2):405--408, 1984.

\bibitem[Dol10]{dolgachevbook}
I.~Dolgachev
\newblock {\em Topics in Classical Algebraic Geometry.}
\newblock Available at \texttt{http://www.math.lsa.umich.edu/\~{}idolga/topics.pdf}

\bibitem[Far10]{farkasevenspin}
G.~Farkas.
\newblock The birational type of the moduli space of even spin curves.
\newblock {\em Adv. Math.}, 223(2):433--443, 2010.

\bibitem[FV09]{farkasverrainterm}
G.~Farkas and A.~Verra.
\newblock The intermediate type of certain moduli spaces of curves.
\newblock arXiv:0910.3905, preprint 2009.

\bibitem[GK10]{grkr}
S.~Grushevsky and I.~Krichever.
\newblock Integrable discrete {S}chr\"odinger equations and a characterization
  of {P}rym varieties by a pair of quadrisecants.
\newblock {\em Duke Math. J.}, 2010.
\newblock to appear.

\bibitem[GSM10]{grsmtwopoint}
Samuel Grushevsky and Riccardo Salvati~Manni.
\newblock The vanishing of two-point functions for three-loop superstring
  scattering amplitudes.
\newblock {\em Comm. Math. Phys.}, 294(2):343--352, 2010.

\bibitem[HM98]{hamobook}
J.~Harris and I.~Morrison.
\newblock {\em Moduli of curves}, volume 187 of {\em Graduate Texts in
  Mathematics}.
\newblock Springer-Verlag, New York, 1998.

\bibitem[Igu72]{igusabook}
J.-I. Igusa.
\newblock {\em Theta functions}.
\newblock Springer-Verlag, New York, 1972.
\newblock Die Grundlehren der mathematischen Wissenschaften, Band 194.

\bibitem[Igu81]{igusachristoffel}
J.-I. Igusa.
\newblock Schottky's invariant and quadratic forms.
\newblock In {\em E. {B}. {C}hristoffel ({A}achen/{M}onschau, 1979)}, pages
  352--362. Birkh\"auser, Basel, 1981.

\bibitem[Iza91]{izadiC-C}
E.~Izadi.
\newblock Fonctions th\^eta du second ordre sur la {J}acobienne d'une courbe
  lisse.
\newblock {\em Math. Ann.}, 289(2):189--202, 1991.

\bibitem[Mum66]{mumfordequations1}
D.~Mumford.
\newblock On the equations defining abelian varieties. {I}.
\newblock {\em Invent. Math.}, 1:287--354, 1966.

\bibitem[Mum67]{mumfordequations3}
D.~Mumford.
\newblock On the equations defining abelian varieties. {III}.
\newblock {\em Invent. Math.}, 3:215--244, 1967.

\bibitem[Mum07]{mumfordbooktheta2}
D.~Mumford.
\newblock {\em Tata lectures on theta. {II}}.
\newblock Modern Birkh\"auser Classics. Birkh\"auser Boston Inc., Boston, MA,
  2007.
\newblock Jacobian theta functions and differential equations, With the
  collaboration of C. Musili, M. Nori, E. Previato, M. Stillman and H. Umemura,
  Reprint of the 1984 original.

\bibitem[Poo94]{poor}
C.~Poor.
\newblock The hyperelliptic locus.
\newblock {\em Duke Math. J.}, 76(3):809--884, 1994.

\bibitem[Run93]{runge1}
B.~Runge.
\newblock On {S}iegel modular forms. {I}.
\newblock {\em J. Reine Angew. Math.}, 436:57--85, 1993.

\bibitem[SM94a]{smlevel2}
R.~Salvati~Manni.
\newblock Modular varieties with level {$2$} theta structure.
\newblock {\em Amer. J. Math.}, 116(6):1489--1511, 1994.

\bibitem[SM94b]{smsubrings}
R.~Salvati~Manni.
\newblock On the projective varieties associated with some subrings of the ring
  of {T}hetanullwerte.
\newblock {\em Nagoya Math. J.}, 133:71--83, 1994.

\bibitem[SMT93]{top}
R.~Salvati~Manni and J.~Top.
\newblock Cusp forms of weight {$2$} for the group {$\Gamma\sb 2(4,8)$}.
\newblock {\em Amer. J. Math.}, 115(2):455--486, 1993.

\bibitem[vGvdG86]{vgvdg}
B.~van Geemen and G.~van~der Geer.
\newblock Kummer varieties and the moduli spaces of abelian varieties.
\newblock {\em Amer. J. Math.}, 108(3):615--641, 1986.

\bibitem[Wel86]{weltersC-C}
G.~Welters.
\newblock The surface {$C-C$} on {J}acobi varieties and 2nd order theta
  functions.
\newblock {\em Acta Math.}, 157(1-2):1--22, 1986.

\bibitem[Zaa95]{zaal}
C.~Zaal.
\newblock Explicit complete curves in the moduli space of curves of genus
  three.
\newblock {\em Geom. Dedicata}, 56(2):185--196, 1995.

\bibitem[Zaa99]{zaal2}
C.~G. Zaal.
\newblock A complete surface in {$M\sb 6$} in characteristic {$>2$}.
\newblock {\em Compositio Math.}, 119(2):209--212, 1999.

\end{thebibliography}

\end{document}